\documentclass[a4paper,11pt]{article}
\usepackage{amsmath, amsfonts, amscd, amssymb, amsthm, enumerate, xypic}

\def\alt{\mathrm{alt}}

\DeclareMathOperator{\id}{\operatorname{id}}

\DeclareMathOperator{\Mat}{\operatorname{M}}
\DeclareMathOperator{\Hom}{\operatorname{Hom}}
\DeclareMathOperator{\Mata}{\operatorname{A}}
\DeclareMathOperator{\End}{\operatorname{End}}

\DeclareMathOperator{\NT}{\operatorname{NT}}
\DeclareMathOperator{\GL}{\operatorname{GL}}
\DeclareMathOperator{\Ker}{\operatorname{Ker}}
\DeclareMathOperator{\Vect}{\operatorname{span}}
\DeclareMathOperator{\im}{\operatorname{Im}}

\DeclareMathOperator{\Sp}{\operatorname{Sp}}
\DeclareMathOperator{\rk}{\operatorname{rk}}
\DeclareMathOperator{\Pf}{\operatorname{Pf}}
\DeclareMathOperator{\codim}{\operatorname{codim}}
\renewcommand{\setminus}{\smallsetminus}
\renewcommand{\epsilon}{\varepsilon}

\def\F{\mathbb{F}}

\def\R{\mathbb{R}}

\renewcommand{\L}{\mathbb{L}}

\def\calA{\mathcal{A}}
\def\calB{\mathcal{B}}

\def\calH{\mathcal{H}}

\def\calL{\mathcal{L}}
\def\calM{\mathcal{M}}

\def\calP{\mathcal{P}}

\def\calS{\mathcal{S}}
\def\calT{\mathcal{T}}

\def\calV{\mathcal{V}}
\def\calW{\mathcal{W}}
\def\calX{\mathcal{X}}
\def\calY{\mathcal{Y}}

\def\lcro{\mathopen{[\![}}
\def\rcro{\mathclose{]\!]}}

\theoremstyle{definition}

\theoremstyle{plain}
\newtheorem{theo}{Theorem}

\newtheorem{cor}[theo]{Corollary}
\newtheorem{lemme}[theo]{Lemma}

\theoremstyle{plain}

\theoremstyle{remark}

\title{On affine spaces of alternating matrices with constant rank}
\author{Cl\'ement de Seguins Pazzis\footnote{Universit\'e de Versailles Saint-Quentin-en-Yvelines, Laboratoire de Math\'ematiques
de Versailles, 45 avenue des Etats-Unis, 78035 Versailles cedex, France}
\footnote{e-mail address: dsp.prof@gmail.com}}

\begin{document}

\thispagestyle{plain}

\maketitle
\begin{abstract}
Let $\F$ be a field, and $n \geq r>0$ be integers, with $r$ even.
Denote by $\Mata_n(\F)$ the space of all $n$-by-$n$ alternating matrices with entries in $\F$.
We consider the problem of determining the greatest possible dimension for an affine subspace of
$\Mata_n(\F)$ in which every matrix has rank equal to $r$ (or rank at least $r$). Recently Rubei \cite{Rubei}
has solved this problem over the field of real numbers. We extend her result to all fields with large enough cardinality.
Provided that $n \geq r+3$ and $|\F|\geq \min\bigl(r-1,\frac{r}{2}+2\bigr)$, we also determine the affine subspaces of rank $r$ matrices in $\Mata_n(\F)$ that have the greatest possible dimension, and we point to difficulties for the corresponding problem in the case $n\leq r+2$.
\end{abstract}




\section{Introduction}

Let $\F$ be a field (possibly of characteristic $2$). Let $V$ be a vector space over $\F$ with finite dimension $n$, and $r$ be an even integer in $\lcro 0,n\rcro$. A bilinear form $b$ on $V$ is called alternating whenever $b(x,x)=0$ for all $x \in V$
(if the characteristic of $\F$ is not $2$, this means that $b$ is skew-symmetric, otherwise these notions are distinct), and it is called
symplectic when it is alternating and non-degenerate.
We denote by $\calA^2(V)$ the vector space of all alternating bilinear forms on $V$, and consider the following three problems:
\begin{enumerate}[(1)]
\item What is the greatest possible dimension $d^{=}_r(V)$ for an affine subspace of $\calA^2(V)$ in which every form has rank $r$?
\item What is the greatest possible dimension $d^{\geq}_r(V)$ for an affine subspace of $\calA^2(V)$ in which every form has rank at least $r$?
\item What is the greatest possible dimension $d^{\leq}_r(V)$ for an affine subspace of $\calA^2(V)$ in which every form has rank at most $r$?
\end{enumerate}
In each case, we might also inquire about the structure of the spaces that attain the greatest possible dimension, but this is very difficult in general. Problem (3) has been solved in \cite{dSPaffinesym} over all fields, including an explicit description of the spaces that attain the greatest possible dimension. In the recent \cite{Rubei}, Elena Rubei has solved problem (1) for arbitrary $n$ and $r$ but only over the field of real numbers
and by using specific properties of this field. Naturally, the above problems can also be stated as problems on subspaces of alternating matrices, and it convenient to display the examples in matrix fashion. Thus, we will denote by $\Mata_n(\F)$ the space of all $n$-by-$n$ alternating matrices (i.e.\ the square matrices $A \in \Mat_n(\F)$ such that $X^TAX=0$ for all $X \in \F^n$).

Following a remark of Roy Meshulam \cite{Meshulamsymmetric} for the corresponding problems in spaces of linear operators from one vector space to another, we will see shortly that problems (1) and (2) are intimately connected with so-called trivial spectrum spaces of endomorphisms.
An endomorphism $u$ of $V$ has \textbf{trivial spectrum} if $\Sp_\F(u) \subseteq \{0\}$, i.e.\ it has no non-zero eigenvalue
in the field $\F$ (but is allowed to have nonzero eigenvalues in algebraic extensions of $\F$). In particular, nilpotent endomorphisms
have trivial spectrum, but the converse is not true in general.
Followingly, a \emph{linear} subspace of $\End(V)$ is said to have trivial spectrum when all its elements have trivial spectrum.
We refer to \cite{Quinlan,dSPgivenrank,dSPlargeaffinenonsingular,dSPAtkinsontoGerstenhaber} for past work on such spaces. We mention in particular the following important result, which
generalizes a famous result of Gerstenhaber on spaces of nilpotent matrices \cite{Gerstenhaber}:

\begin{theo}[See \cite{Quinlan,dSPgivenrank}]
The greatest possible dimension for a trivial spectrum linear subspace of $\End(V)$ is
$\dbinom{\dim V}{2}\cdot$
\end{theo}

In \cite{dSPlargeaffinenonsingular}, the spaces that attain the greatest possible dimension, which we call the \textbf{optimal trivial spectrum subspaces}, were related to the classification of (potentially non-symmetric) non-isotropic bilinear forms over $\F$
(provided that $|\F|>2$).

Now, say that we have an affine subspace $\calB$ of $\calA^2(V)$ in which every element has rank $n$, i.e.\ is symplectic. Take
an arbitrary $s_0 \in \calB$. Assign to every bilinear form $s$ on $V$ the unique endomorphism
$u \in \End(V)$ such that $s(x,y)=s_0(x,u(y))$ for all $(x,y)\in V^2$.
This way, we create an isomorphism $\Phi$ from $\calA^2(V)$ to the space $\calA_{s_0}$ of all $s_0$-alternating endomorphisms of $V$
(an endomorphism $u$ is $s_0$-alternating whenever $s_0(x,u(x))=0$ for all $x \in V$).
Now, let $u \in \calA_{s_0}$. Then $u-\id$ is non singular if and only if $s_0- \Phi^{-1}(u)$ is symplectic.
By a simple homogeneity argument, this yields that $u$ has trivial spectrum if and only each form in the affine subspace
$s_0+\F \Phi^{-1}(u)$ is symplectic. Hence, denoting by $\overrightarrow{\calB}$ the translation vector space of $\calB$, we gather that
$\Phi(\overrightarrow{\calB})$ has trivial spectrum. And conversely,
if we have a symplectic form $s_0$ on $V$ together with a linear subspace $L$ of $\calA_{s_0}$
with trivial spectrum, then $s_0+\Phi^{-1}(L)$ is an affine subspace of symplectic forms on $V$, with the same dimension as $L$.

Consequently, solving the case $n=r$ in problems (1) and (2) (they are equivalent in that case)
amounts to determining the greatest possible dimension for a trivial spectrum linear subspace of $\calA_s$
when $s$ is an arbitrary symplectic form on an $n$-dimensional vector space (with $n$ even).
If $\F$ is algebraically closed, this can be obtained as a consequence of corresponding results on nilpotent linear subspaces
of $\calA_s$ (see \cite{structuredGerstenhaber1} for fields with characteristic other than $2$, and \cite{structuredGerstenhaber3} for fields with characteristic $2$). In this note, our first major result is a generalization to all fields of large enough cardinality:

\begin{theo}\label{theo:trivialspectrum}
Let $V$ be an $\F$-vector space of even dimension $2n$, and $s$ be a symplectic form on $V$.
Assume that $|\F| \geq 2n-1$.
Let $\calS$ be a trivial spectrum linear subspace of $\calA_s$.
Then $\dim \calS \leq n(n-1)$.
\end{theo}

The optimality of this result (apart from the restriction on the cardinality of $\F$) is illustrated in the following example. Let $\calW$ be an optimal trivial spectrum subspace of
$\Mat_n(\F)$ (e.g.\ the space $\NT_n(\F)$ of all strictly upper-triangular matrices).
Then one sees that the set of all matrices of the form
$$\begin{bmatrix}
A & B \\
0 & A^T
\end{bmatrix}, \quad \text{with $A \in \calW$ and $B \in \Mata_n(\F)$,}$$
represents, in the standard basis of $\F^{2n}$, a space of $s$-alternating endomorphisms
for the symplectic form $s$ whose Gram matrix in that basis equals the standard symplectic matrix
$$\begin{bmatrix}
0 & I_n \\
-I_n & 0
\end{bmatrix}.$$
Obviously, this is a trivial spectrum space with dimension $2\dbinom{n}{2}=n(n-1)$.

As an immediate corollary of Theorem \ref{theo:trivialspectrum} and of this example, we obtain:

\begin{theo}\label{theo:diminv}
Let $V$ be a vector space of even dimension $2n$, and $s$ be a symplectic form on $V$.
Assume that $|\F| \geq 2n-1$. Then the greatest possible dimension for
an affine subspace of $\calA^2(V)$ consisting of symplectic forms is $n(n-1)$.
\end{theo}

\begin{theo}\label{theo:dimmin}
Let $V$ be a vector space of dimension $n$, and $r=2s$ be an even integer in $\lcro 0,n\rcro$.
Assume that $|\F| \geq n-1$ if $n$ is even, and $|\F| \geq n-2$ if $n$ is odd.
Then
$$d^{\geq}_r(V)=s(s-1)+r(n-r)+\frac{(n-r)(n-r-1)}{2}=\dim \calA^2(V)-s^2.$$
\end{theo}

\begin{theo}\label{theo:dim}
Let $V$ be a vector space of dimension $n$, and $r=2s$ be an even integer in $\lcro 0,n\rcro$.
Assume that $|\F| \geq \max\left(r-1,2+\frac{r}{2}\right)$.
Then
$$d^{=}_r(V)=\begin{cases}
s\,(n-s-1) & \text{if $n \neq r+1$} \\
s\,(s+1) & \text{if $n=r+1$.}
\end{cases}$$
\end{theo}

Let us immediately show that the stated dimensions can be attained in those theorems.
We start with the second one. Letting $\calM$ be an affine subspace of non-singular matrices of $\Mat_s(\F)$ with dimension
$\frac{s(s-1)}{2}$ (for example, $I_s+\NT_s(\F)$), we take
$$\widetilde{\calM}^{(n)}_{\alt}:=\left\{
\begin{bmatrix}
A & B & C \\
-B^T & [0] & [0] \\
-C^T & [0] & [0]
\end{bmatrix} \mid
A \in \Mata_s(\F), \; B \in \calM, \;  C \in \Mat_{s,n-r}(\F)\right\} \subseteq \Mata_n(\F).$$
It is easily seen that $\dim \widetilde{\calM}^{(n)}_\alt=s(s-1)+s(n-2s)=s(n-s-1)$. And because of the assumption that all the matrices in $\calM$
are non-singular it is also clear that $\widetilde{\calM}^{(n)}_\alt$ has constant rank $2s$.

Next, if $n=r+1$ we can take an affine subspace
$\calH \subseteq \Mata_{n-1}(\F)$ with dimension $s(s-1)$ and whose elements are all non-singular (e.g.\ given by the previous example), and then take the space
$$\calH^+:=
\left\{\begin{bmatrix}
H & C \\
-C^T & 0
\end{bmatrix} \mid H \in \calH, \; C \in \F^{n-1}\right\} \subseteq \Mata_n(\F).$$
Clearly all the matrices in $\calH^+$ have rank at least $n-1$, and since they are alternating their rank is even, and hence at most $n-1$.
And finally $\dim \calH^+=\dim \calH + (n-1)=s(s+1)$.

Finally, let us start from an affine subspace $\calH$ of $\Mata_r(\F)$ in which every element is non-singular, and with $\dim \calH=s(s-1)$,
and consider the space
$$\overline{\calH}^{(n)}:=\left\{\begin{bmatrix}
H & C \\
-C^T & D
\end{bmatrix} \mid H \in \calH, \; C\in \Mat_{r,n-r}(\F),\;   D \in \Mata_{n-r}(\F) \right\} \subseteq \Mata_n(\F).$$
Again, it is clear that all the matrices in $\overline{\calH}^{(n)}$ have rank at least $r$, and the dimension of the space is
$$\dim \overline{\calH}^{(n)}=\dim \Mata_n(\F)-\codim_{\Mata_r(\F)} \calH=\dbinom{n}{2}-s^2.$$
Hence, it only remains to prove the inequalities
$$d^{\geq}_r(V) \leq \dim \calA^2(V)-s^2$$
and
$$d^{=}_r(V) \leq \begin{cases}
s\,(n-s-1) & \text{if $n \neq r+1$} \\
s\,(s+1) & \text{if $n=r+1$.}
\end{cases}$$
The proof of the first one will be deduced from Theorem \ref{theo:diminv} thanks to
Meshulam's method from \cite{Meshulamsymmetric} (Section \ref{section:dimmin}). The case $n=r$ in the second one is given by Theorem \ref{theo:diminv}.
For the other cases, we will use the same strategy as in Rubei's article \cite{Rubei} to deduce the inequality from Theorem \ref{theo:diminv}.

As an offspring of our method, in Section \ref{section:maxspaces} we will obtain the following
partial result on the affine spaces that attain the greatest possible dimension in problem (1):

\begin{theo}\label{theo:dimmax}
Let $V$ be a vector space of dimension $n$, and $r=2s>0$ be an even integer with $n>r+2$.
Assume that $|\F| \geq \max\left(r-1,2+\frac{r}{2}\right)$. Let $\calS$ be an affine subspace of $\calA^2(V)$ in which every element has rank $r$, and with
$\dim \calS=s\,(n-s-1)$.

Then there exists a basis of $V$ and an affine subspace $\calM \subseteq \GL_s(\F)$ with dimension $\frac{s(s-1)}{2}$ such that
$\calS$ is represented in the said basis by $\widetilde{\calM}_\alt^{(n)}$.
Moreover, the equivalence class\footnote{Two subsets $\calX$ and $\calY$ of $\Mat_{n,p}(\F)$ are called equivalent
when there exist invertible matrices $P \in \GL_n(\F)$ and $Q \in \GL_p(\F)$ such that $\calY=P \calX Q$, meaning that
$\calX$ and $\calY$ represent the same set of linear mappings in a different choice of bases.}
of $\calM$ is uniquely determined by $\calS$.
\end{theo}

As stated earlier, the classification, up to equivalence, of the affine subspaces of $\Mat_s(\F)$ included with $\GL_s(\F)$ and with dimension $\frac{s(s-1)}{2}$ is well understood when $|\F|>2$ (it is connected to the one of non-isotropic quadratic forms over $\F$).

Conversely, it is easily checked that if $\calM$ and $\calM'$ are equivalent affine subspaces of $\Mat_s(\F)$ then
$\widetilde{\calM}_\alt^{(n)}$ and $\widetilde{\calM'}_\alt^{(n)}$ are congruent affine subspaces of $\Mata_n(\F)$.

If $n=r+2$, there are examples that do not fit the result of Theorem \ref{theo:dimmax}: for instance, one can take an affine subspace $\calH$ of dimension $s(s+1)$ of
$\Mata_{r+1}(\F)$ in which every matrix has rank $r$, and consider the affine space of all matrices of the form
$$\begin{bmatrix}
H & [0]_{(r+1) \times 1} \\
[0]_{1 \times (r+1)} & 0
\end{bmatrix} \quad \text{with $H \in \calH$.}$$
An inspection of the proof of Theorem \ref{theo:dimmax} makes us worry that the special cases $n \in \{r+1,r+2\}$ are far more difficult than the case $n>r+2$, and we prefer to abstain from going any further.

\section{Technical lemmas}

Our proof techniques essentially rely on basic block-matrix results from the theory of vector spaces of bounded rank matrices.
Chiefly, we will use the following result, which we call the Flanders-Atkinson lemma, and several of its corollaries.
We refer to \cite{AtkinsonPrim}, \cite{FLR} and Section 2 of \cite{LLD1} for various proofs and versions of it.

\begin{lemme}[Flanders-Atkinson lemma]
Let $n,p,r$ be integers with $0<r \leq \min(n,p)$. Assume that $|\F|>r$.
Let $J_r:=\begin{bmatrix}
I_r & 0 \\
0 & 0
\end{bmatrix}$ and $M=\begin{bmatrix}
A & C \\
B & D
\end{bmatrix}$ belong to $\Mat_{n,p}(\F)$, with $A \in \Mat_r(\F)$ and so on.
If $\rk(sJ_r+t M) \leq r$ for all $(s,t) \in \F^2$, then $D=0$ and $B A^kC=0$ for every integer $k \geq 0$.
\end{lemme}

\begin{cor}\label{affinecor}
Let $n,p,r$ be integers with $0<r \leq \min(n,p)$. Assume that $|\F|>r+1$.
Let $J_r:=\begin{bmatrix}
I_r & 0 \\
0 & 0
\end{bmatrix}$ and $M=\begin{bmatrix}
A & C \\
B & D
\end{bmatrix}$ belong to $\Mat_{n,p}(\F)$, with $A \in \Mat_r(\F)$ and so on.
If $\rk(J_r+t M) \leq r$ for all $t \in \F$, then 
$$D=0 \quad \text{and} \quad \forall k \geq 0, \; B A^kC=0.$$
\end{cor}

\begin{proof}[Proof of Corollary \ref{affinecor}]
One simply uses the
assumption $|\F|>r+1$ to gather that all the $(r+1) \times (r+1)$ minors of $x J_r+y M$ vanish for all $(x,y)\in \F^2$
(take such a minor as a function of $(x,y)$, and note that it is a homogeneous polynomial of degree $r+1$ that vanishes
at more than $r+1$ points of the projective line of $\F^2$). Then one applies the Flanders-Atkinson lemma.
\end{proof}

\begin{cor}\label{cor:FLAalt}
Let $n,r$ be integers with $0<r \leq n$ and $r$ even. Assume that $|\F|> \frac{r}{2}+1$.
Let $J_K:=\begin{bmatrix}
K & 0 \\
0 & 0
\end{bmatrix}$ and $M=\begin{bmatrix}
A & B \\
-B^T & D
\end{bmatrix}$ belong to $\Mata_n(\F)$, with $A$ and $K$ in $\Mata_r(\F)$, and so on. Assume that $K$ is invertible.
If $J_K+t M$ has rank at most $r$ for all $t \in \F$, then $D=0$ and $B^T K^{-1}(AK^{-1})^k B=0$ for every integer $k \geq 0$.
\end{cor}

\begin{proof}
There is a subtlety here as we have assumed that $|\F|> \frac{r}{2}+1$ instead of
$|\F|> r+1$. Of course, if the latter holds then it suffices to apply Corollary \ref{affinecor} after right-multiplying with
$K^{-1} \oplus I_{n-r}$.
Now, assume that $\F$ is finite. We can choose a field extension $\mathbb{L}$ of $\F$ such that $|\L|>r+1$.
Then we claim that $J_K+t M$ has rank at most $r$ for all $t \in \mathbb{L}$. The key to obtain this is to use the Pfaffian
(denoted by $\Pf$) instead of the determinant! First of all, it is critical to note that if an alternating matrix $N$ has rank at least $2s$ for some integer $s \geq 1$, then one of its principal $2s \times 2s$ submatrices is invertible. This can be proved as follows: first of all,
we can write $r'=\rk N$ and pick a direct factor of the radical of $N$ that is spanned by vectors of the standard basis. The
corresponding $r' \times r'$ submatrix of $N$ is then invertible. And then one proceeds by downward induction by using the development
of the Pfaffian along the last row/column.

Now, take an arbitrary subset $I$ of $\lcro 1,n\rcro$ with cardinality $r+2$, and for $N \in \Mata_n(\F)$
denote by $N_{I,I}$ the corresponding principal submatrix.
The mapping $t \in \L \mapsto \Pf((J_K+t M)_{I,I})$ is a polynomial function of degree at most $\frac{r+2}{2}$, and it vanishes everywhere on $\F$. Since $|\F|> \frac{r}{2}+1$, it also vanishes everywhere on $\L$.
Varying $I$ shows, thanks to the previous remark, that $J_K+tM$ has rank at most $r$ for all $t \in \L$.

Then, applying Corollary \ref{cor:FLAalt} in $\L$ yields the claimed result.
\end{proof}

As a consequence of the conclusion ``$D=0$" in the Flanders-Atkinson lemma, we also have the following result
in terms of subspaces of linear mappings:

\begin{cor}\label{cor:minimalFA}
Let $U$ and $V$ be finite-dimensional vector spaces, and $\calS$ be a linear subspace of $\Hom(U,V)$.
In $\calS$, take an element $u_0$ of maximal rank $r$, and assume that $|\F|>r$.
Then every element of $\calS$ maps $\Ker u_0$ into $\im u_0$.
\end{cor}

Finally, we recall a consequence of the main result of \cite{dSPlargeaffinerankbelow}, to be used in Section \ref{section:maxspaces}:

\begin{theo}\label{theo:affinemax}
Let $p$ and $r$ be positive integers with $1 \leq r \leq p$. Assume that $|\F|>2$.
Let $\calT$ be an affine subspace of $\Mat_{r,p}(\F)$ in which every matrix has rank $r$, and
assume that $\codim_{\Mat_{r,p}(\F)} \calT \leq \frac{r(r+1)}{2}\cdot$
Then, there exists an affine subspace $\calM$ of $\Mat_r(\F)$ in which every matrix is invertible, with
$\dim \calM=\frac{r(r-1)}{2}$, and such that $\calT$ is equivalent to the space
$$\widetilde{\calM}^{(p)}:=\left\{\begin{bmatrix}
B & C
\end{bmatrix} \mid (B,C) \in \calM \times \Mat_{r,p-r}(\F)\right\}.$$
Moreover, the equivalence class of $\calM$ is uniquely determined by $\calT$.
\end{theo}

\begin{proof}
When $r \geq 2$, Theorem \ref{theo:affinemax} is the special case $n=r$ in theorem 3 of \cite{dSPlargeaffinerankbelow}. We wish to note that the result remains true in the special case $r=1$. In that case, it is essentially an obvious result on the (affine) hyperplanes of $\Mat_{1,n}(\F)$ that do not contain the zero vector: simply, take such a hyperplane $\calT$, choose $x_1 \in \calT$ and then extend this vector to a basis of $\Mat_{1,n}(\F)$ by using a basis of the translation vector space of $\calT$. This shows that $\calT$ is equivalent to the space of all row vectors with first entry equal to $1$,
and hence the conclusion is satisfied for $\calM:=\{1\}$ (note that the uniqueness statement is obvious in that case).
A close inspection of the proof of theorem 3 of \cite{dSPlargeaffinerankbelow} also reveals that if $n=r$ then the case $r=1$ need not be discarded.
\end{proof}

The proof of the main theorem of \cite{dSPlargeaffinerankbelow} is difficult: it relies upon the classification of the optimal trivial spectrum subspaces of $\Mat_r(\F)$, as given in \cite{dSPlargeaffinenonsingular}.

\section{Trivial spectrum linear subspaces of alternating endomorphisms}\label{section:trivialspectrum}

Here we prove Theorem \ref{theo:trivialspectrum} by induction on $n$.
The case $n=0$ is trivial, and now we assume that $n \geq 1$.
The idea is to use the operator-vector duality, in a way that is reminiscient to the basic idea of \cite{dSPAtkinsontoGerstenhaber}.
For $x \in V$, we consider the linear operator
$$\widehat{x} : u \in \calS \mapsto u(x) \in V.$$
This yields a linear subspace
$$\widehat{\calS}=\{\widehat{x} \mid x \in V\} \subseteq \Hom(\calS,V).$$
Now, let $x \in V \setminus \{0\}$. Since $\calS$ has trivial spectrum, we have
$\F x \cap \calS x=\{0\}$.
But we also have $\calS x \subseteq \{x\}^{\bot_s}$ because $\calS$ consists of $s$-alternating operators.
And finally $x \in \{x\}^{\bot_s}$ since $s$ is alternating. Therefore
$$\F x \oplus \calS x \subseteq \{x\}^{\bot_s}.$$
It follows in particular that $\dim(\calS x) \leq 2n-2$, that is $\rk \widehat{x} \leq 2n-2$.

Now, we take $x \in V \setminus \{0\}$ such that $\widehat{x}$ has the greatest possible rank in $\widehat{\calS}$, denoted by $r'$.
In particular $r' \leq 2n-2$ and Corollary \ref{cor:minimalFA} yields that
$\widehat{\calS}$ maps every vector of $\Ker \widehat{x}$ into $\im \widehat{x}$ (this is where the assumption $|\F|\geq 2n-1$ comes into play).
So, set
$$\calS':=\Ker \widehat{x}=\{u \in \calS : u(x)=0\}.$$
Now, even though we might have $r'<2n-2$, we can always embed $\im \widehat{x}=\calS x$ into a linear hyperplane $W$ of $\{x\}^{\bot_s}$
such that $\F x \oplus W=\{x\}^{\bot_s}$. In particular $W$ is $s$-regular.
The previous remark yields that $\im u \subseteq W$ for every $u \in \calS'$.
Because each $u \in \calS'$ is $s$-alternating and hence $s$-selfadjoint, we deduce that the elements of $\calS'$
vanish everywhere on $W^{\bot_s}$. And finally $V=W\oplus W^{\bot_s}$ because $W$ is $s$-regular.

Hence, by restricting to $W$ we obtain a linear injection from $\calS'$
to a linear subspace of $\calA_{s_W}$, where $s_W$ stands for the symplectic form induced by $s$ on $W^2$.
The range of the said injection obviously has trivial spectrum.
Hence by induction (because $|\F| \geq 2n-1 \geq 2(n-1)-1$) we find
$$\dim \calS' \leq (n-1)(n-2).$$
By the rank theorem, we conclude that
$$\dim \calS=\dim \calS'+\dim \calS x \leq (n-1)(n-2)+(2n-2)=(n-1)\,n.$$
This completes the proof of Theorem \ref{theo:trivialspectrum}.

\section{Affine subspaces of alternating forms with bounded rank}\label{section:dimmin}

Now, we prove Theorems \ref{theo:dimmin} and \ref{theo:dim}.
To start with, we let $\calS$ be an affine subspace of $\calA^2(V)$ in which every element has rank \emph{at least} $r$, and we assume that
$|\F| \geq n-1$ if $n$ is even, and $|\F| \geq n-2$ if $n$ is odd.

The case $n=r$ has already been dealt with in Theorem \ref{theo:trivialspectrum}, so we assume $n>r$.
By downward induction on $r$, we can assume that $\calS$ actually contains an element $s_0$ of rank $r$
(indeed, the dimension stated in the first part of Theorem \ref{theo:dim} is a non-increasing functions of $r$, as seen
by its second expression, and the cardinality assumption on $|\F|$ garantees that $|\F| \geq r_0-1$ for the least possible rank
$r_0$ of the elements of $\calS$).

Let us then take an arbitrary basis of $V$ in which the last $n-r$ vectors span the radical of $s_0$, and let us represent the elements of $\calS$
in that basis: for each $b \in \calA^2(V)$ we have a corresponding alternating matrix
$$M(b)=\begin{bmatrix}
A(b) & B(b) \\
-B(b)^T & D(b)
\end{bmatrix} \quad \text{where $A(b) \in \Mata_r(\F), B(b) \in \Mat_{r,n-r}(\F), D(b) \in \Mata_{n-r}(\F)$.}$$
Note that $B(s_0)=0$ and $D(s_0)=0$.
Set
$$K:=A(s_0) \in \GL_r(\F) \cap \Mata_r(\F)$$
and consider the affine subspace
$$\calT:=\{b \in \calS : B(b)=0 \; \text{and} \; D(b)=0\}$$
(which contains $s_0$).
Every element of $\calT$ has rank at least $r$, so
$A(\calT)$ is an affine subspace of matrices of $\Mata_r(\F)$ with constant rank $r$.
By Theorem \ref{theo:diminv}, we have
$$\dim \calT=\dim A(\calT) \leq s(s-1),$$
and we conclude by the rank theorem for affine mappings that
$$\dim \calS \leq \dim \calT+r(n-r)+\dim \Mata_{n-r}(\F)\leq \dim \Mata_n(\F)-s^2.$$
Thus Theorem \ref{theo:dimmin} is now proved.

\vskip 3mm
In the remainder, we turn to the proof of Theorem \ref{theo:dim}: we assume that every element of $\calS$ has rank $r$, and we modify the cardinality assumption on $\F$:
here we assume that $|\F| \geq \max\left(r-1,2+\frac{r}{2}\right)$.

Then the argument is slightly different but we start again from the previous block form. Now, we introduce the translation vector space $S$ of $\calS$ and we apply Corollary
\ref{cor:FLAalt}, which uses
the assumption that $|\F|>\frac{r}{2}+1$ and that every element of $\calS$ has rank at most $r$.
This yields
\begin{equation}\label{eq:FLAbasic}
\forall b \in S, \; D(b)=0 \quad \text{and} \quad B(b)^T K^{-1} B(b)=0.
\end{equation}
From the first identity, we get
$$\dim \calS=\dim \calT+\dim B(S).$$
If $n=r+1$ this is clearly enough to conclude.

Now, assume that $n \geq r+2$.
Then we shall prove that $\dim B(S) \leq (n-r)s$, which will be enough to conclude.
To obtain this, we interpret the second identity in \eqref{eq:FLAbasic} as meaning that the range of $K^{-1} B(b)$
is totally $K$-singular for all $b \in S$. And the expected result will come from the following lemma:

\begin{lemme}\label{lemma:existlagrangian}
Let $U$ and $U'$ be finite-dimensional vector spaces, with $\dim U>1$, and $b$ be a symplectic form on $U'$.
Let $\calV \subseteq \Hom_\F(U,U')$ be a linear subspace in which every element has its range totally $b$-singular.
Then
$$\dim \calV \leq \frac{(\dim U)(\dim U')}{2}\cdot$$
Moreover, if $\dim U>2$ and $\dim \calV=\frac{(\dim U)(\dim U')}{2}$ then
$\calV=\Hom_\F(U,\calL)$ for some Lagrangian\footnote{
For a symplectic form $b$ on a vector space $U'$, a Lagrangian is a totally $b$-singular subspace of $U'$ with dimension
$\frac{\dim U'}{2}\cdot$} $\calL$ of $(U',b)$.
\end{lemme}

\begin{proof}
Set $p:=\dim U$ and $q:=\dim U'$ for convenience, and $r:=\frac{q}{2}\cdot$

If $\dim \calV x \leq r$ for all $x \in \F^p \setminus \{0\}$, then we directly
have $\dim \calV \leq p r$ by taking a basis of $U$.
Now, assume that $\dim \calV x_0 > r$ for some $x_0 \in U \setminus \{0\}$.
Consider then the subspace $\calV':=\{u \in \calV : u(x_0)=0\}$.
Let $x \in U$. Let $u' \in \calV'$ and $u \in \calV$. Then $(u+u')(x)$ is $b$-orthogonal to $(u+u')(x_0)=u(x_0)$, and $u(x)$ is $b$-orthogonal to $u(x_0)$. Hence $u'(x)$ is $b$-orthogonal to $u(x_0)$. Thus $\calV' x \subseteq (\calV x_0)^{\bot_b}$. By taking a basis of a complementary subspace of $\F x_0$ in $U$, we deduce that
$$\dim \calV' \leq (p-1) \dim (\calV x_0)^{\bot_b}.$$
Hence by the rank theorem
$$\dim \calV = \dim (\calV x_0)+\dim \calV' \leq \dim U'+(p-2) \dim (\calV x_0)^{\bot_b} \leq 2r+(p-2)\, r=pr.$$
Note that the last inequality is sharp if $p-2>0$.
Now, assume that $p>2$ and $\dim \calV=pr$. Hence by the last remark we must have $\dim \calV x \leq r$ for all $x \in U \setminus \{0\}$.
Take a basis $(x_1,\dots,x_p)$ of $U$.
Since $\dim \calV=pr$, the linear mapping
$\psi : v \in \calV \mapsto (v(x_1),\dots,v(x_p)) \in \calV x_1 \times \cdots \times \calV x_p$ must be surjective and
all the $\calV x_i$'s must have dimension $r$. By the surjectivity of $\psi$, we deduce that
$\calV x_1,\dots,\calV x_p$ are pairwise $b$-orthogonal.
But since they all have dimension $r$ we have $\calV x_i=(\calV x_j)^{\bot_b}$ for all distinct $i,j$ in $\lcro 1,p\rcro$, and since
$p \geq 3$ we obtain that all the $\calV x_i$'s are equal. Their common value is then a Lagrangian $\calL$ of $(U',b)$.
Hence $\calV \subseteq \Hom_\F(U,\calL)$ and we conclude by $\calV=\Hom_\F(U,\calL)$ since the dimensions are equal.
\end{proof}

This complete the proof of Theorem \ref{theo:dim}.

\section{Affine subspaces of alternating forms with constant rank and large dimension}\label{section:maxspaces}

Here, we prove Theorem \ref{theo:dimmax}.
We come back to the situation of the previous section.
Now, we assume that $n>r+2$ and that $\calS$ has the critical dimension $s(n-s-1)$, with all the forms in $\calS$ of rank $r$.
With the previous proof, we gather in particular that $\dim B(S)=s(n-r)$ and $\dim \calT=s(s-1)$, and we can use the last statement of Lemma \ref{lemma:existlagrangian} to obtain a Lagrangian $\calL$ of $\F^r$ for the symplectic form $(X,Y) \mapsto X^TKY$, such that $K^{-1} B(\calS)$ is the set of all matrices
of $\Mat_{r,n-r}(\F)$ with range included in $\calL$.

\vskip 3mm
\noindent \textbf{Step 1: Proving that $\calL$ is totally $A(s)$-singular for all $s \in \calS$} \\
Let us take an arbitrary linear section $\theta : B(S) \rightarrow S$ of the projection of $S$ onto $B(S)$.
By Corollary \ref{cor:FLAalt}, we find that
\begin{equation}\label{eq:FLA}
\forall N \in B(S), \; (K^{-1}N)^T A(\theta(N))\, K^{-1} N=0.
\end{equation}
For $i \in \lcro 1,n-r\rcro$, put
$$f_i : X_i \in \calL \mapsto A\left(\theta\left(\begin{bmatrix}
0 & \cdots &  0 &  K X_i  & 0 & \cdots&  0
\end{bmatrix}\right)\right)$$
where the $K X_i$ vector appears on the $i$-th column.

Let $k \in \lcro 1,n-r\rcro$. Choose distinct elements $i,j$ in  $\lcro 1,n-r\rcro \setminus \{k\}$ (this is possible because $n > r+2$).
Applying \eqref{eq:FLA} at the $(i,j)$-spot yields
$$\forall (X_1,\dots,X_{n-r})\in \calL^{n-r},\quad
X_i^T\, \sum_{\ell=1}^{n-r} f_\ell(X_\ell)\, X_j=0.$$
Replacing $X_k$ with $0$ and subtracting the two identities thereby obtained, we get
$$\forall (X_i,X_j,X_k) \in \calL^3, \; X_i^T f_k(X_k) X_j=0.$$
Fixing $X_k$ and varying $X_i$ and $X_j$, we obtain that $\calL$ is totally $f_k(X_k)$-singular.
Varying $k$ and $X_k$ then yields that $\calL$ is totally $A(\theta(N))$-singular for all $N \in B(S)$.

As this yields for every choice of the section $\theta$, we conclude that $\calL$ is totally $A(b)$-singular for all
$b \in S$. Because it is also totally $K$-singular, we conclude that $\calL$ is totally $A(b)$-singular for all $b \in \calS$.

\vskip 3mm
\noindent \textbf{Step 2: Preparing the reduced form} \\
Now we refine the choice of the starting basis so that $\calL=\{0\} \times \F^s$
and $K=\begin{bmatrix}
0 & I_s \\
-I_s & 0
\end{bmatrix}$.
In that case every matrix of $B(S)$ has its rows zero starting from the $(s+1)$-th.
Note that $B(\calS)=B(S)$ because $B(s_0)=0$. And finally the first step shows that
every matrix $A(b)$ with $b \in \calS$ has its lower-right $s \times s$ block equal to zero.
Therefore, for every $b$ of $\calS$ we now have
$$M(b)=\begin{bmatrix}
U(b) & R(b) \\
-R(b)^T & [0]
\end{bmatrix} \; \text{for some $U(b) \in \Mata_s(\F)$ and some $R(b) \in \Mat_{s,n-s}(\F)$.}$$
Next, note that $R(\calS)$ is an affine subspace of $\Mat_{s,n-s}(\F)$ and
$\dim R(\calS) \geq \dim \calS-\frac{s(s-1)}{2}=\dim \Mat_{s,n-s}(\F)-\frac{s(s+1)}{2}$.
Moreover, for all $b \in \calS$, we see that $\rk M(b) \leq s+\rk R(b)$ (erase the first $s$ columns) and hence $\rk R(b) \geq s$.
Hence every matrix in $R(\calS)$ has rank $s$.

\vskip 3mm
\noindent \textbf{Step 3: Concluding the reduction} \\
Noting that $\frac{r}{2}+2>2$, we see that Theorem \ref{theo:affinemax} applies to $R(\calS)$.
This yields $Q \in \GL_{s}(\F)$ and $Q' \in \GL_{n-s}(\F)$ (actually, one could take $Q=I_s$)
and an affine subspace $\calM$ of nonsingular matrices of $\Mat_s(\F)$, with dimension $\frac{s(s-1)}{2}$,
such that $Q R(\calS) Q'=\widetilde{\calM}^{(n-s)}$.
Then $P:= Q \oplus (Q')^T$ belongs to $\GL_n(\F)$ and $P M(\calS) P^T \subseteq \widetilde{\calM}^{(n)}_\alt$.
As both spaces have dimension $s(n-s-1)$, we conclude that
$$P M(\calS) P^T=\widetilde{\calM}^{(n)}_\alt.$$
This completes the proof of the first statement in Theorem \ref{theo:dimmax}.

\vskip 3mm
\noindent \textbf{Step 4: Uniqueness} \\
It remains to prove that the equivalence class of $\calM$ is uniquely determined by $\calS$.
Towards this end, we choose a basis $(e_1,\dots,e_n)$ of $V$ in which $\calS$ is represented by $\widetilde{\calM}^{(n)}$.
Note that $\calL:=\Vect(e_{s+1},\dots,e_n)$ is totally singular for all the forms in $\calS$.
The key is to prove that it is the sole such space with dimension $n-s$.
So, we take an arbitrary subspace $\calL' \subseteq V$ with dimension $n-s$ and which is totally singular for all the forms in $\calS$.
Then clearly $\calL'$ is also totally singular for all the forms in the translation vector space $S$ of $\calS$.

First of all, we prove that $\dim (\calL \cap \calL')\geq n-s-1$. To see this, note from the definition of $\widetilde{\calM}^{(n)}_\alt$
that $S$ contains every alternating form on $V$ whose radical includes $\calL$. If some pair $(x,y)\in (\calL')^2$ is linearly independent modulo
$\calL$, then among such forms we can choose one $b$ such that $b(x,y) \neq 0$, contradicting a previous statement
(indeed, extend $(x,y,e_{s+1},\dots,e_n)$ into a basis $(x,y,e_{s+1},\dots,e_n,f_1,\dots,f_{s-2})$ of $V$,
take the dual basis $(\varphi_1,\dots,\varphi_n)$ and consider the alternating form $b : (z,z') \mapsto \varphi_1(z)\,\varphi_2(z')-\varphi_1(z')\,\varphi_2(z)$).
Hence the projection of $\calL'$ onto $V/\calL$ has dimension at most $1$, and we conclude that $\dim (\calL \cap \calL') \geq n-s-1$.

Now, assume that $\dim (\calL \cap \calL')=n-s-1$, so that $\dim(\calL+\calL')=n-s+1$. Note that, for all $x \in \calL \cap \calL'$
and all $b \in S$, the linear form $b(-,x)$ vanishes everywhere on $\calL+\calL'$, whence
$$\dim \{b(-,x) \mid b \in S\} \leq \codim_V(\calL+\calL')=s-1.$$
Observing the last $n-r$ columns in the matrices of $\widetilde{\calM}^{(n)}$, this forces
$(\calL \cap \calL') \cap \Vect(e_{r+1},\dots,e_n)=\{0\}$, and since we are dealing with subspaces of $\calL$ we derive that
$$\dim (\calL \cap \calL') \leq (n-s)-(n-r)=s.$$
But this would yield $n \leq 2s+1$, in contradiction with the assumption $n > r+2$.
This shows that $\calL'=\calL$.

Now we can complete the proof. Let $\calM'$ be an affine subspace of nonsingular matrices of
$\Mat_s(\F)$, with dimension $\frac{s(s-1)}{2}$, and assume that $\widetilde{\calM'}^{(n)}_\alt$ represents $\calS$ in some basis
$(e'_1,\dots,e'_n)$. Then $\calL'=\Vect(e'_{s+1},\dots,e'_n)$ is totally singular for all the elements of $\calS$, and hence
it equals $\Vect(e_{s+1},\dots,e_n)$ by the first part of this step. Therefore the matrix $P$ of coordinates of
$(e'_1,\dots,e'_n)$ in $(e_1,\dots,e_n)$ reads
$$P=\begin{bmatrix}
P_1 & [0] \\
[?]_{(n-s) \times s} & P_3
\end{bmatrix} \quad \text{with $P_1 \in \GL_s(\F)$ and $P_3 \in \GL_{n-s}(\F)$.}$$
Moreover $P^T \widetilde{\calM}^{(n)}_\alt P=\widetilde{\calM'}^{(n)}_\alt$. Extracting the upper-right blocks, this yields $P_1^T \widetilde{\calM}^{(n-s)} P_3=\widetilde{\calM'}^{(n-s)}$. By the uniqueness statement in Theorem \ref{theo:affinemax} we conclude that $\calM$ is equivalent to $\calM'$. This completes the proof of Theorem \ref{theo:dimmax}.

\section{Open questions and comments}

In this final section, we wish to make some comments on the previous results and their limitations.
We have already pointed to the fact that Theorem \ref{theo:dimmax} has no immediate adaptation to the case $n=r+2$, and
the assumption $n>r+2$ was extensively used in our proof. The case $n=r+1$ seems to be even more difficult.
What is remarkable in our proof is that, although $\calT$ was shown early on to be an affine subspace of invertible elements of $\Mata_r(\F)$
with the greatest possible dimension (that is, $s(s-1)$), we did not require any classification of such spaces, i.e.\ of the solution to the case
$n=r$. We suspect that such a solution is unavoidable for the case $n=r+1$, and for $n=r+2$ we also have the feeling that it might also be.

At this point of course we have not formulated any conjecture on the form of the optimal spaces for $n=r$, i.e.\ of the
affine subspaces of $\Mata_r(\F)$ with dimension $s(s-1)$ in which all the elements are invertible. Actually, those spaces are known if $\F$ is algebraically closed with characteristic other than $2$. In that case indeed, trivial spectrum subspaces coincide with nilpotent subspaces, and hence one of the main results from \cite{structuredGerstenhaber1} (theorem 1.9 there) yields that, for a symplectic form $s_0$ on a $2s$-dimensional vector space, and for every trivial spectrum subspace $\calS$ of $\calA_{s_0}$, there exists an $s_0$-symplectic basis $(e_1,\dots,e_s,f_1,\dots,f_s)$ of $V$ in which $\calS$ is represented by the space of all matrices of the form
$$\begin{bmatrix}
A & B \\
[0]_{s \times s} & A^T
\end{bmatrix} \quad \text{with $A \in \NT_s(\F)$ and $B \in \Mata_s(\F)$}.$$
Hence, it follows that, up to congruence, the sole affine subspace of $\Mata_r(\F)$ that has dimension $s(s-1)$ and constant rank $r$
is
$$\left\{\begin{bmatrix}
[0]_{s \times s} & I_s+A \\
-(I_s+A)^T & B
\end{bmatrix} \mid A \in \NT_s(\F) \; \text{and}\; B \in \Mata_s(\F)\right\}.$$

Now, an apparently reasonable conjecture would generalize the above to trivial spectrum subspaces, as
follows: Instead of $\NT_s(\F)$, we take an arbitrary trivial spectrum (linear) subspace $\calV$ of $\Mat_s(\F)$ with dimension $\frac{s(s-1)}{2}$. Then
the corresponding affine subspace of $\Mata_r(\F)$ with constant rank $r$ and dimension $s(s-1)$ is
$$\left\{\begin{bmatrix}
[0]_{s \times s} & I_s+A \\
-(I_s+A)^T & B
\end{bmatrix} \mid A \in \calV \; \text{and}\; B \in \Mata_s(\F)\right\}.$$
The conjecture would state that every affine subspace of  $\Mata_r(\F)$ with constant rank $r$ and dimension $s(s-1)$ is congruent to a space
of this form. Yet, in working on the present article, we discovered that this conjecture is wrong, which is seen by observing the critical
case where $r=4$. Indeed, if this conjecture held true, then every $2$-dimensional affine subspace of non-singular matrices of $\Mata_4(\F)$
would have a rank $2$ matrix in its translation vector space (as seen from the lower-right $B$ cell).
Yet, this is false, as we shall now see.

Indeed, in $\Mata_4(\F)$ the invertibility of matrices is controlled by the Pfaffian, which is a hyperbolic quadratic form of rank $6$.
And to construct a counter-example is then easy, as it suffices to construct a plane $\calP$ of $\Mata_4(\F)$ that does not go through zero and
that can be embedded in a $3$-dimensional linear subspace of $\Mata_4(\F)$ in which all the \emph{non-zero} elements are invertible.
This is easy if $\F$ has its u-invariant greater than $2$ (otherwise it is not possible!).
We will also assume that $\chi(\F) \neq 2$ for convenience, but fields with characteristic $2$ could also be encompassed.
So, assume that there is a nonisotropic quadratic form $q$ over $\F$ with rank $3$.
Then $q \bot (-q)$ is a hyperbolic form of rank $6$, and hence it is equivalent to the $4$-by-$4$ Pfaffian.
This yields that the latter has a $3$-dimensional linear subspace $W$ on which the restriction of the Pfaffian is nonisotropic,
to the effect that every nonzero element of $W$ has rank $4$. To conclude, we simply pick a $2$-dimensional affine subspace $\calP$ of $W$
that does not contain $0$. Then of course $\calP$ consists only of rank $4$ matrices, and its translation vector space is also included in $W$
and hence contains no rank $2$ matrix. Hence $\calP$ is an exception to the above conjecture.

In practice the above abstract construction can be used to give explicit counterexamples. For example, for $\F=\R$ we take $q=\langle 1,1,1\rangle$,
which prompts us to consider the space $W$ of all real skewsymmetric matrices of the form
$$A(x,y,z)=\begin{bmatrix}
0 & x & y & z \\
-x & 0 & z & -y \\
-y & -z & 0 & x \\
-z & y & -x & 0
\end{bmatrix}$$
with $(x,y,z) \in \R^3$ (note that $\Pf(A(x,y,z))=x^2+y^2+z^2$).
And a counterexample is obtained by considering the plane $\calP=\{A(x,y,1) \mid (x,y)\in \R^2\}$.

This example makes us worry that obtaining the equivalent of Theorem \ref{theo:dimmax} for $n=r$ should be very difficult.
Finally, noting that the case $n=r+1$ in Theorem \ref{theo:dimmax} is close to the equivalent of deriving Theorem \ref{theo:affinemax}
from the main result of \cite{dSPlargeaffinenonsingular}, it can hardly be hoped that anything short of a complete solution to the case $n=r$
is required for the case $n=r+1$, and that, even so, deriving the case $n=r+1$ from the case $n=r$ should be very difficult.

Finally, even if all those problems were solved, there will still remain the issue of the cardinality assumptions in all the results we have proved so far.
The techniques we used made these assumptions unavoidable. Yet, for the equivalent problems on rectangular matrices there have been results
that hold for all fields \cite{dSPlargeaffinerankbelow,dSPlargeaffinenonsingular} (with the exception of the field with two elements). Achieving such a goal in the present context would require a complete revolution in the
methods, and so far we have failed to come up with any valid idea on how to tackle small finite fields.

\end{document}